\documentclass[11pt,reqno,oneside]{amsart}
\usepackage{amscd}
\usepackage{amsmath}
\usepackage{latexsym}
\usepackage{amsfonts}
\usepackage{amssymb}
\usepackage{amsthm}
\usepackage{graphicx}
\usepackage{verbatim}
\usepackage{mathrsfs}
\usepackage{enumerate}
\usepackage{hyperref}
\usepackage{fullpage}
\usepackage{xcolor}

\theoremstyle{plain}
\theoremstyle{plain}\newtheorem{theorem}{Theorem}[section]
\theoremstyle{plain}\newtheorem{lemma}[theorem]{Lemma}
\theoremstyle{plain}
\theoremstyle{plain}\newtheorem{proposition}[theorem]{Proposition}
\theoremstyle{plain}\newtheorem{remark}{Remark}[section]

\newcommand{\R}{\mathbb{R}}
\newcommand{\be}{\begin{equation}}
\newcommand{\ee}{\end{equation}}
 \newcommand{\ba}{\begin{aligned}}
 \newcommand{\ea}{\end{aligned}}

  \newcommand{\ben}{\begin{enumerate}}
   \newcommand{\een}{\end{enumerate}}

\newcommand{\Rmnum}[1]{\expandafter\@slowromancap\romannumeral #1@}
\allowdisplaybreaks

\begin{document}
%%%%%%%%%%%%%%%%%%%%%%%%%%%%%%%%%%%%%%%%%%%%%%%%%%%%%%%%%%%%%%%%%%%%%%%%%%%%%%%%%%%%%%%%%%%%%%%%%%%%
\title{An Extension of Riesz Transform}

\author[H. Yu, Q. Jiu]{Huan Yu$^{1}$, Quansen Jiu$^{2}$}

\address{$^1$  School of Applied Science, Beijing Information Science and Technology University, Beijing, 100192, P.R.China}

\email{yuhuandreamer@163.com}

\address{$^2$ School of Mathematical Sciences, Capital Normal University, Beijing, 100048, P.R.China}

\email{jiuqs@cnu.edu.cn}

\date{\today}

\begin{abstract}

In this paper,  we  consider  the following singular integral
\begin{equation*}
T_jf(x)=K_j*f(x), K_j(x)=\frac{x_j}{|x|^{n+1-\beta}},
\end{equation*}
where $x\in \R^n,  0\le\beta<n, j=1,2,\cdots, n$. When $\beta=0$, it corresponds to the Riesz transform. We will make an estimate the $L^q (1<q<\infty)$ norm of $T_jf$,  which holds  uniformly for $0\le\beta<\frac{n(q-1)}{q}$. In particular, when $\beta=0$, the strong $(q,q)$ type estimate of the Riesz transform for $1<q<\infty$ is recovered from the obtained estimate.

\end{abstract}
\smallskip
\maketitle%\centerline{$^{1}$Institute of Applied Physics and Computational Mathematics}
%\centerline{Beijing  100088, P. R. China}
%\centerline{\it Email: chen\_qionglei@iapcm.ac.cn}

%\centerline{\scshape Huan Yu}
%\centerline{$^{2}$Institute of Applied Physics and Computational Mathematics}
%\centerline{Beijing  100088, P. R. China}
%\centerline{\it Email: yuhuandreamer@163.com}

%%%%%%%%%%

\section{Introduction and Main Results}

Given  $f(x)\in L^1(\R^n)\cap L^q(\R^n)$ with $ 1\le q<\infty$, we consider the singular integral
\begin{equation}\label{T-operator}
T_jf(x)=K_j*f(x), K_j(x)=\frac{x_j}{|x|^{n+1-\beta}}, 0<\beta<n, j=1,2,\cdots, n.
\end{equation}
Let $\hat{f}$ be the Fourier transform of $f$  defined as
\begin{equation*}
\hat f(y)=\int_{\R^n} e^{2\pi ix\cdot y}f(x) dx.
\end{equation*}
Then the Fourier transform of $T_jf$ is
\begin{equation}\label{T-operator2}
\widehat{T_jf}(\xi)=\gamma_\beta\frac{\xi_j}{|\xi|^{\beta+1}},  0<\beta<n, \gamma_\beta=i\pi^{n/2-\beta}\frac{\Gamma(\frac{\beta+1}{2})}{\Gamma(\frac{n+1-\beta}{2})},
\end{equation}
Formally, when $\beta=0$, $T_jf$ is the well-known Riesz transform. In fact, it holds (see \cite{[stein]})
$$
\lim_{\beta\to 0+}\int_{\R^n} \frac{x_j}{|x|^{k+n-\beta}}\hat{\varphi}(x) dx=\lim_{\epsilon\to 0+}\int_{|x|\ge\epsilon} \frac{x_j}{|x|^{k+n}}\hat{\varphi}(x) dx,
$$
where $\varphi(x)\in \mathcal{S}(\R^n)$ which is the Schwartz space.  In view of the Riesz potential estimate (see \cite{[stein]}), it is direct to deduce that, for $1<p< q<\infty$,
\begin{equation}\label{E931}
\|T_jf\|_q\le \|\int_{\R^n}\frac{|f(y)|}{|x-y|^{n-\beta}}\,dy\|_q\le C(\beta)\|f\|_p,  \frac 1q=\frac 1p-\frac \beta n, 0<\beta<n.
\end{equation}
However, the constant $C(\beta)$ on the right side of \eqref{E931} depends on $\beta$ in general and  is unbounded as $\beta\to 0$.
A natural question is whether one can  obtain an uniform $L^q-$ estimate of $T_jf$ with respect to $\beta>0$ such that  the strong $(q,q)$ type estimate of the Riesz transform  can be recovered when $\beta\to 0$.  We will answer this question in this paper.

  Here and in what follows, $\|f\|_q$ means the $L^q(\R^n)$ norm of $f$. To simplify the presentation, we omit subscript of $T_j$ and $K_j$ and write \eqref{T-operator} as
\begin{equation}\label{T-operator1}
Tf(x)=K*f(x), K(x)=\frac{x_j}{|x|^{n+1-\beta}}, 0<\beta<n
\end{equation}
for any $j=1,2,\cdots, n$.

 Then  our main result  can be  stated as
\begin{theorem}\label{uni-est3}
Let $f\in L^1(\R^n)\cap L^q(\R^n)$, $1< q <\infty$. Then there exists a constant $C=C(n)$ independent of $\beta$ such that
\begin{equation}\label{T1-111-}
\|Tf\|_q\leq C(\|f\|_q+\|f\|_p+\frac{\beta^{\frac{n(q-1)}{q}}}{(n(q-1)-\beta q)^{\frac 1q}}\|f\|_1)
\end{equation}
for  $0\le \beta<\frac{n(q-1)}{q}$ and $1<p\le q<\infty$ satisfying $\frac 1q=\frac 1p-\frac{\beta}{n}$. Consequently, it holds
\begin{equation}\label{T1-110}
\|Tf\|_q\leq C(\|f\|_q+L(\beta)\|f\|_1)
\end{equation}
for $0\le \beta<\frac{n(q-1)}{q}$, where $L(\beta)=\frac{\beta^{\frac{n(q-1)}{q}}}{(n(q-1)-\beta q)^{\frac 1q}}+\frac{\beta q}{(q-1)n}$.
\end{theorem}

\begin{remark}
 It is addressed that the constant $C$ on the right of \eqref{T1-111-} and \eqref{T1-110} does not depend on $\beta$ and hence  the $(p,p)$ type estimate of the Riesz transform can be recovered from  \eqref{T1-111-} and \eqref{T1-110} respectively when $\beta\to 0$.
\end{remark}
To prove  Theorem \ref{uni-est3}, we split the singular integral \eqref{T-operator1} into two parts: the part near the origin denoted by $T_1f$ and the one apart from the origin denoted by $T_2f$. The estimate on $\|T_2f\|_q$  is easy to obtain (see Lemma \ref{uni-est1}).
  The key part is to estimate $\|T_1f\|_q$. We will use  the refined Calderon-Zygmund decomposition to overcome new difficulties  encountered in the estimate of  $\|T_1b\|_q$ (see proof of Lemma  \ref{uni-est32+}). Moreover, we have
\begin{theorem}\label{uni-est30}
Let $f\in L^1(\R^n)\cap L^q(\R^n)$, $1\le q<\infty$. Then there exists a constant $C=C(n)$ independent of $\beta$ such that
\begin{equation}\label{T1-116}
m\{x: |T_1f|>t\}\le C(\frac{\|f\|_1}{t}+\frac{\|f\|^q_1}{t^q})
\end{equation}
for any $t>0$ and $q\ge 1$ satisfying $\frac 1q=1-\frac{\beta}{n}$.
\end{theorem}
Here $m(A)$ means the Lebesgue measure of a set $A\subset \R^n$. When $\beta\to 0$ it concludes that $q=1$ and the weak $(1,1)$ estimate of the Riesz transform can be recovered from \eqref{T1-116}.

The kind of singular integral \eqref{T-operator} or \eqref{T-operator1} appears in  the generalized surface quasi-geostrophic (SQG) equation which reads as
\begin{equation}\label{SQG}
\left\{\ba
&\omega_{t}+u\cdot\nabla \omega =0, ~(x,t)\in \R^{2}\times\R^{+},\\
&u=\nabla^{\perp}(-\Delta)^{-1+\alpha}\omega, \\
&\omega(x,0)=\omega_0.\ea\ \right.
\end{equation}
Here $0\le \alpha\le\frac 12$ and $\nabla^\perp=(-\partial_{x_2}, \partial_{x_1})$.  The unknown functions  $\omega=\omega(x,t)$ and $u=u(x,t)=(u_1(x,t),u_2(x,t))$ are  related by $\eqref{SQG}_2$ which can be expressed as
\begin{equation}\label{Int}
u(x)=\int_{\R^2} \frac{(x-y)^\perp}{|x-y|^{2+2\alpha}}\omega(y) dy.
\end{equation}
Here $x^\perp=(-x_2, x_1)$ and the singular integral \eqref{Int} means the principle value one.  When $\alpha=0$, \eqref{SQG} corresponds to the two-dimensional incompressible Euler equations. In this case, the unknown functions  $\omega=\omega(x,t)$ and $u=u(x,t)$ are the vorticity and the velocity field respectively. When $\alpha=\frac12$, \eqref{SQG} corresponds to the surface quasi-geostrophic (SQG) equation which describes a famous approximation model of the nonhomogeneous fluid flow in a rapidly rotating 3D half-space (see \cite{[CMT]},\cite{[P87]}). When $0<\alpha<\frac12$, it is called the generalized (or modified) SQG equation.
In the case $0<\alpha\le \frac12$, the unknown functions  $\omega=\omega(x,t)$ and $u=u(x,t)$ stand for potential temperature and velocity field respectively. It is noted that when $\alpha=\frac 12$ the relation between $u=u(x,t)$ and  $\omega=\omega(x,t)$ in \eqref{Int} corresponds to the Riesz transform. When $0<\alpha<\frac 12$, the relation \eqref{Int} is completely similar to  the $T$ operator defined in  \eqref{T-operator1}  with $\beta=1-2\alpha$ and $n=2$. It is clear that $\beta$ will vanish as $\alpha\to \frac12$. In \cite{[JYZ]}, we investigate the approximation of the SQG equation by the generalized SQG equation as $\alpha\to \frac 12$. What's more, when $q=p=2$ in Theorem \ref{uni-est3}, the following result has been established in \cite{[JYZ]}:
\begin{proposition}\label{uni-est}
For $f\in L^1(\R^n)\cap L^2(\R^n)$, it holds that
\begin{equation}\label{T1-10}
\|Tf\|_2\leq C(\|f\|_2+\frac{\beta^{\frac n2}}{\sqrt{n-2\beta}}\|f\|_1),\quad \,  \,\,0<\beta<\frac n2
\end{equation}
for some constant $C=C(n)$ independent of $\beta$.
\end{proposition}
Clearly, Proposition \ref{uni-est} is a particular case of Theorem  \ref{uni-est3}.

The paper is organized as follows. In Section 2, we will present some preliminary estimates which will be needed later. The proof of  Theorem \ref{uni-est3} and Theorem \ref{uni-est30} will be given in Sections 3.

\section{Preliminaries}
\setcounter{section}{2}\setcounter{equation}{0}

Let  $\chi(s)\in C_0^\infty(R)$ be the usual smooth cutting-off function which is defined as
$$
\chi(s)=
\left\{
\begin{array}{ll}
1, & |s|\le 1,\\[3mm]
0, & |s|\ge 2,
\end{array}
\right.
$$
satisfying $|\chi'(s)|\le 2$. Let
\begin{equation}\label{Chi-L}
\chi_\lambda(s)=\chi(\lambda s),
\end{equation}
and define
\begin{eqnarray*}
&& T_1f(x)=K_1*f(x), K_1(x)=K(x)\chi_\beta(|x|),\label{T1-operator}\\
&& T_2f(x)=K_2*f(x), K_2(x)=K(x)(1-\chi_\beta(|x|)).\label{T2-operator}
\end{eqnarray*}
 Then it is clear that the operator $T$ in \eqref{T-operator1} can be written as
\begin{equation}\label{T-j}
T=T_1+T_2.
\end{equation}

The following is a $L^q$-estimate of $T_2$:
\begin{lemma}\label{uni-est1}
 There exists an absolute constant $C>0$ independent of $\beta$ such that for any $1<q<\infty$,
\begin{equation}\label{T2-11}
\|T_2f\|_q\leq C\frac{\beta^{\frac{n(q-1)}{q}}}{(n(q-1)-\beta q)^{\frac 1q}}\|f\|_1,  0<\beta<\frac{n(q-1)}{q};
\end{equation}
\end{lemma}
\begin{proof}[Proof of Lemma \ref{uni-est1}]
Note that
\begin{equation*}
T_2f(x)=\int_{\R^n}\frac{x_j-y_j}{|x-y|^{n+1-\beta}}(1-\chi_\beta(|x-y|))f(y)\,dy.
\end{equation*}
Then direct estimates give
\begin{equation*}\begin{split}
\|T_2f\|_q&\leq\|\int_{|x-y|\geq\frac{1}{\beta}}\frac{1}{|x-y|^{n-\beta}}|f(y)|\,dy\|_q
\\&\leq\|f\|_1\int_{|x-y|\geq\frac{1}{\beta}}(\frac{1}{|x-y|^{q(n-\beta)}}\,dy)^{\frac1q}
\\&\leq C\frac{\beta^{\frac{n(q-1)}{q}}}{(n(q-1)-\beta q)^{\frac 1q}}\|f\|_1
\end{split}
\end{equation*}
for any $0<\beta<\frac{n(q-1)}{q}$.
\end{proof}

Concerning the operator $T_1$, we first prove that it is of type $(2,2)$, which has been shown in \cite{[JYZ]}. For completeness, we give a sketch of proof here.
\begin{lemma}\label{uni-est21}
 There exists a constant $C=C(n)$ independent of $\beta$ such that
\begin{equation}\label{T1-1}
\|T_1f\|_2\leq C\|f\|_2, 0<\beta< n.
\end{equation}
\end{lemma}

\begin{proof}[Proof of Lemma \ref{uni-est21}]

 To prove \eqref{T1-1}, our main target is to  prove that there exists an absolute constant $C>0$ independent $\beta$ such that
\begin{equation}\label{K1-F}
\|\widehat{K_1}(y)\|_{L^\infty}\leq C, \ \ 0<\beta<n.
\end{equation}
Since $\int_{\mathbb{S}^1} K_1(x) ds=0$ (here $\mathbb{S}^1$ is the unit sphere surface in $\mathbb{R}^n$) and $K_1(x)$ is supported on $|x|\le \frac 2\beta$, we have
\begin{equation}
\widehat{K_1}(y)=\int_{\R^n} e^{2\pi ix\cdot y}K_1(x)\,dx=\int_{|x|\leq\frac{2}{\beta}} (e^{2\pi ix\cdot y}-1)K_1(x)\,dx
\end{equation}
Since \eqref{K1-F} is a pointwise estimate, we will estimate $\widehat{K_1}(y)$ by different values of $y$.  If $|y|<\frac{\beta}{2}$,  it is direct to estimate
\begin{equation}\label{s-1}
\begin{split}
|\widehat{K_1}(y)|&\leq C|y|\int_{|x|\leq \frac 2\beta}|x|\frac{1}{|x|^{n-\beta}}\,dx
\\ &\leq \frac{2^\beta}{\beta+1}\beta^{-\beta}.
\end{split}
\end{equation}

Then there exists an absolute constant $C>0$ such that
\begin{equation}\label{K1-F1}
|\widehat{K_1}(y)|\leq C,  0<\beta<n, \ |y|< \frac{\beta}{2}.
\end{equation}

If $\frac{\beta}{2}\leq |y|\leq\beta,$ we rewrite $\widehat{K_1}(y)$ as
\begin{equation*}\begin{split}
\widehat{K_1}(y)&=\int_{|x|<\frac{1}{|y|}} e^{2\pi ix\cdot y}K_1(x)\,dx+\int_{\frac{1}{|y|}\leq|x|\leq\frac{2}{\beta}} e^{2\pi ix\cdot y}K_1(x)\,dx
\\&=\int_{|x|<\frac{1}{|y|}} (e^{2\pi ix\cdot y}-1)K_1(x)\,dx+\int_{\frac{1}{|y|}\leq|x|\leq\frac{2}{\beta}} e^{2\pi ix\cdot y}K_1(x)\,dx
\end{split}
\end{equation*}
Similar to \eqref{s-1}, it deduces
\begin{equation*}\begin{split}
|\int_{|x|<\frac{1}{|y|}} (e^{2\pi ix\cdot y}-1)K_1(x)\,dx|\leq \frac{2^\beta}{\beta+1}\beta^{-\beta}.
\end{split}
\end{equation*}
Moreover, we have
\begin{equation*}\begin{split}
|\int_{\frac{1}{|y|}\leq|x|\leq\frac{2}{\beta}} e^{2\pi ix\cdot y}K_1(x)\,dx|\leq C\frac{2^\beta-1}{\beta}\beta^{-\beta}.
\end{split}
\end{equation*}

Consequently, there exists an absolute constant $C>0$ such that
\begin{equation}\label{K1-F2}
|\widehat{K_1}(y)|\leq C(\frac{2^\beta}{\beta+1}\beta^{-\beta}+\frac{2^\beta-1}{\beta}\beta^{-\beta})\le C, \ 0<\beta<n, \frac{\beta}{2}\le |y|\le\beta.
\end{equation}

If $ |y|>\beta,$ $\widehat{K_1}(y)$ can be divided into
\begin{equation}\label{K1-F31}
\begin{split}
\widehat{K_1}(y)&=\int_{|x|<\frac{1}{|y|}} e^{2\pi ix\cdot y}K_1(x)\,dx+\int_{\frac{1}{|y|}\leq|x|\leq\frac{2}{\beta}} e^{2\pi ix\cdot y}K_1(x)\,dx
\\&=\int_{|x|<\frac{1}{|y|}} (e^{2\pi ix\cdot y}-1)K_1(x)\,dx+\int_{\frac{1}{|y|}\leq|x|\leq\frac{2}{\beta}} e^{2\pi ix\cdot y}K_1(x)\,dx.
\end{split}
\end{equation}
For the first term on the right hand of the above equality, we obtain
\begin{equation}\label{K1-F311}
\begin{split}
|\int_{|x|<\frac{1}{|y|}} (e^{2\pi ix\cdot y}-1)K_1(x)\,dx|&\leq C|y|\int_{|x|<\frac{1}{|y|}}|x|\frac{1}{|x|^{n-\beta}}\,dx
\\&\leq \frac{1}{\beta+1}\beta^{-\beta}.
\end{split}
\end{equation}
For the second term,
we  choose $z=\frac{y}{2|y|^2}$ with $|z|=\frac{1}{2|y|}<\frac{1}{2\beta}$ such that $e^{2\pi iy\cdot z}=-1$
and
\begin{equation*}
\int_{\R^n} e^{2\pi ix\cdot y}K_1(x)\,dx=\frac12\int_{\R^n} e^{2\pi ix\cdot y}(K_1(x)-K_1(x-z))\,dx,
\end{equation*}
so
\begin{equation}\label{K1-F321}
\begin{split}
\int_{\frac{1}{|y|}\leq|x|\leq\frac{2}{\beta}} e^{2\pi ix\cdot y}K_1(x)\,dx&=\frac12\int_{\frac{1}{|y|}\leq|x|\leq\frac{2}{\beta}} e^{2\pi ix\cdot y}(K_1(x)-K_1(x-z))\,dx
\\&-\frac12\int_{\frac{1}{|y|}\leq|x+z|,~ |x|\leq\frac{1}{|y|}} e^{2\pi ix\cdot y}K_1(x)\,dx
\\&+\frac12\int_{|x+z|\leq\frac{1}{|y|},~ |x|\geq\frac{1}{|y|}} e^{2\pi ix\cdot y}K_1(x)\,dx\\
&+\frac12\int_{|x+z|\ge\frac{2}{\beta}} e^{2\pi ix\cdot y}K_1(x) dx
\\&\equiv I+J+K+L.
\end{split}
\end{equation}

To estimate the term $I$, we have
\begin{equation}\label{F321-I}
\begin{split}
I&=\int_{\frac{1}{|y|}\leq|x|<\frac{1}{\beta},~|x-z|\leq\frac{1}{\beta}}(\frac{x}{|x|^{n+1-\beta}}-\frac{x-z}{|x-z|^{n+1-\beta}})e^{2\pi ix\cdot y}\,dx
\\&+\int_{\frac{1}{\beta}\leq|x|\leq\frac{2}{\beta},~|x-z|\leq\frac{1}{\beta}}(\frac{x}{|x|^{n+1-\beta}}\chi_\beta(x)-\frac{x-z}{|x-z|^{n+1-\beta}})e^{2\pi ix\cdot y}\,dx
\\&+\int_{\frac{1}{|y|}\leq|x|<\frac{1}{\beta},~|x-z|\geq\frac{1}{\beta}}(\frac{x}{|x|^{n+1-\beta}}-\frac{x-z}{|x-z|^{n+1-\beta}}\chi_\beta(x-z))e^{2\pi ix\cdot y}\,dx
\\&+\int_{\frac{1}{\beta}\leq|x|\leq\frac{2}{\beta},~|x-z|\geq\frac{1}{\beta}}(\frac{x}{|x|^{n+1-\beta}}\chi_\beta(x)-\frac{x-z}{|x-z|^{n+1-\beta}}\chi_\beta(x-z))e^{2\pi ix\cdot y}\,dx
\\&=I_1+I_2+I_3+I_4.
\end{split}
\end{equation}

We first estimate $I_2.$ Thanks to $|x-z|\geq|x|-|z|\geq\frac{1}{\beta}-\frac{1}{2|y|}\geq\frac{1}{2\beta},$ one has
\begin{equation}\label{I-2}
\begin{split}
|I_2|&\leq\int_{\frac{1}{\beta}\leq|x|\leq\frac{2}{\beta}}\frac{1}{|x|^{n-\beta}}\,dx
+\int_{\frac{1}{2\beta}\leq|x-z|\leq\frac{1}{\beta}}\frac{1}{|x-z|^{n-\beta}}\,dx
\\&\leq C\frac{2^\beta-1}{\beta}\beta^{-\beta}+C\frac{1-2^{-\beta}}{\beta}\beta^{-\beta}.
\end{split}
\end{equation}

Then thanks to  $|x|=|x-z+z|\geq|x-z|-|z|\geq\frac{1}{\beta}-\frac{1}{2|y|}\geq\frac{1}{2\beta},$ $I_3$ is estimated as follows.
\begin{equation}\label{I-3}
\begin{split}
|I_3|&\leq\int_{\frac{1}{2\beta}\leq|x|\leq\frac{1}{\beta}}\frac{1}{|x|^{n-\beta}}\,dx
+\int_{\frac{1}{\beta}\leq|x-z|\leq\frac{2}{\beta}}\frac{1}{|x-z|^{n-\beta}}\,dx
\\&\leq C\frac{1 -2^{-\beta}}{\beta}\beta^{-\beta}+C\frac{2^\beta-1}{\beta}\beta^{-\beta}.
\end{split}
\end{equation}

The term $I_4$ is directly estimated as
\begin{equation}\label{I-4}\begin{split}
|I_4|&\leq\int_{\frac{1}{\beta}\leq|x|\leq\frac{2}{\beta}}\frac{1}{|x|^{n-\beta}}\,dx
+\int_{\frac{1}{\beta}\leq|x-z|\leq\frac{2}{\beta}}\frac{1}{|x-z|^{n-\beta}}\,dx
\\&\leq C\frac{2^\beta-1}{\beta}\beta^{-\beta}.
\end{split}
\end{equation}

Now we deal with $I_1.$  Note that
\begin{equation*}
\partial_i(\frac{x}{|x|^{n+1-\beta}})=\frac{\overrightarrow{e}_i}{|x|^{n+1-\beta}}
+(-n-1+\beta)\frac{xx_i}{|x|^{n+3-\beta}},~ i=1,2,...,n.
\end{equation*}
In this case, since $|x-z|\geq|x|-|z|\geq2|z|-|z|\geq|z|,$ by Taylor expansion, one has
\begin{equation*}\begin{split}
&|\frac{x-z}{|x-z|^{n+1-\beta}}-\frac{x}{|x|^{n+1-\beta}}|\\ \le&
|\sum_{i=1}^n\big(\frac{z_i\overrightarrow{e}_i}{|x-z|^{n+1-\beta}}+
(-n-1+\beta)\frac{(x-z)(x_i-z_i)z_i}{|x-z|^{n+3-\beta}}\big)|+C\sum_{k=2}^\infty \frac{|z|^k}{k!|x-z|^{n+k-\beta}}.
\end{split}
\end{equation*}
Consequently,
\begin{equation}\label{I-1}
\begin{split}
|I_1|&\leq (n+2-\beta)|z|\int_{|z|\leq|x-z|<\frac{1}{\beta}<\infty}\frac{1}{|x-z|^{n+1-\beta}}\,dx+
C\sum_{k=2}^\infty \int_{|z|\leq|x-z|<\frac{1}{\beta}<\infty}\frac{|z|^k}{k!|x-z|^{n+k-\beta}}\,dx
\\&\leq C\frac{|z|^\beta}{1-\beta}\leq C\frac{\beta^{-\beta}}{1-\beta}.
\end{split}
\end{equation}
Substituting \eqref{I-2}-\eqref{I-1} into \eqref{F321-I} yields
\begin{equation}\label{I-E}
\begin{split}
|I|&=|\int_{\frac{1}{|y|}\leq|x|<\frac{1}{\beta},~|x-z|\leq\frac{1}{\beta}}(\frac{x}{|x|^{n+1-\beta}}-\frac{x-z}{|x-z|^{n+1-\beta}})e^{2\pi ix\cdot y}\,dx|\\
&\le  C(\frac{2^\beta-1}{\beta}\beta^{-\beta}+\frac{1-2^{-\beta}}{\beta}\beta^{-\beta}+\frac{\beta^{-\beta}}{1-\beta})
\end{split}
\end{equation}
for some absolute constant $C>0$.

Concerning the term $J$, thanks to $|x|\geq|x+z|-|z|\geq 2|z|-|z|\geq |z|$, one has
\begin{equation}\label{J-E}
\begin{split}
|J|&\leq\int_{|z|\leq|x|\leq2|z|}\frac{1}{|x|^{n-\beta}}\,dx
\\&\leq\frac{1-2^{-\beta}}{\beta}\beta^{-\beta}.
\end{split}
\end{equation}

Concerning the term $K$, thanks to  $|x|\leq|x+z|+|z|\leq 2|z|+|z|\leq 3|z|$, one has
\begin{equation}\label{K-E}
\begin{split}
|K|&\leq\int_{2|z|\leq|x|\leq3|z|}\frac{1}{|x|^{n-\beta}}\,dx
\\&\leq\frac{(\frac32)^{\beta}-1}{\beta}\beta^{-\beta}.
\end{split}
\end{equation}
Concerning the term $L$, thanks to $\frac{2}{\beta}\ge |x|\ge |x+z|-|z|\ge \frac{2}{\beta}-\frac{1}{2\beta}=\frac{3}{2\beta}$, one has
\begin{equation}\label{L-E}
\begin{split}
|L|&\leq\int_{\frac{3}{2\beta}\leq|x|\leq\frac{2}{\beta}}\frac{1}{|x|^{n-\beta}}\,dx
\\&\leq\frac{1}{\beta}[(\frac{3}{2\beta})^\beta-(\frac{2}{\beta})^\beta].
\end{split}
\end{equation}
Substituting \eqref{I-E}-\eqref{L-E} into \eqref{K1-F321}, we obtain that there exists an absolute constant $C>0$ such that
\begin{equation}\label{K12}
|\int_{\frac{1}{|y|}\leq|x|\leq\frac{2}{\beta}} e^{2\pi ix\cdot y}K_1(x)\,dx|\le C.
\end{equation}
In view of \eqref{K1-F311}, \eqref{K12} and \eqref{K1-F31}, there exists an absolute constant $C>0$ such that
\begin{equation}\label{K1-F3}
|\widehat{K_1}(y)|\le C, \ \ 0<\beta<n, |y|>\beta.
\end{equation}
Combining  \eqref{K1-F1}, \eqref{K1-F2} with \eqref{K1-F3}, we finish the proof of \eqref{K1-F}. Applying \eqref{K1-F}, one has
\begin{equation*}\label{T1-1+}\begin{split}
\|T_1f\|_{L^2}=\|\widehat{K_1}\widehat{f}\|_{L^2}\le C\|\widehat{f}\|_{L^2}=C\|f\|_{L^2}.\\
\end{split}
\end{equation*}
 Hence \eqref{T1-1} is proved and the proof of the lemma is complete.
\end{proof}

The following is a Marcinkiewicz interpolation theorem (see \cite{[stein]}).

\begin{lemma}\label{912}
Suppose that $1<r\le\infty$. Suppose that the following holds:

 (1) $T$ is a sub-additive mapping from $L^1(\R^n)+L^r(\R^n)$ to the space of measurable functions on $\R^n$:
 $$
 |T(f+g)(x)|\le |Tf(x)|+|Tg(x)|.
 $$

 (2) $T$ is of weak-type $(1,1)$:
 $$
 m\{x: |Tf(x)|>t\}\le \frac{A_1}{t}\|f\|_1,  f\in L^1(\R^n).
 $$

 (3) $T$ is of weak-type $(r,r)$:
 $$
 m\{x: |Tf(x)|>t\}\le (\frac{A_r}{t}\|f\|_r)^r,  f\in L^r(\R^n)
 $$
 if $r<\infty$ or
 $$
 \|Tf\|_\infty\le A_\infty \|f\|_\infty
 $$
 if $r=\infty$.

 Then $T$ is of type $(p,p)$ for all $1<p<r$, that is,
 $$
 \|Tf\|_p\le A_p\|f\|_p, f\in L^p(\R^n)
 $$
 for all  $1<p<r$, where $A_p$ depends only on $A_1, A_r, p$ and $r$.
 \end{lemma}
The proof of  \ref{912} is referred to \cite{[stein]} and we omit it here.

\section{Proof of Main Results}
\setcounter{section}{3}\setcounter{equation}{0}

In this section, we give the proof of Theorem \ref{uni-est3} and Theorem \ref{uni-est30}.

\begin{proof}[Proof of Theorem \ref{uni-est3}]

In view of Lemma \ref{uni-est1}, to prove Theorem \ref{uni-est3}, it suffices to prove

\begin{lemma}\label{uni-est32+}
Let $f\in L^1(\R^n)\cap L^q(\R^n)$. Then  there exists a constant $C=C(n)$ independent of $\beta$ such that
\begin{equation*}\label{T1-111}
\|T_1f\|_q\leq C(\|f\|_q+\|f\|_p+\frac{\beta^{\frac{n(q-1)}{q}}}{(n(q-1)-\beta q)^{\frac 1q}}\|f\|_1)
\end{equation*}
for any $1<p\le q<\infty$ satisfying $\frac 1q=\frac 1p-\frac{\beta}{n}$.
\end{lemma}

Now we prove Lemma \ref{uni-est32+}. Given $t>0$, according to the  cube decomposition procedure, $\R^n$ can be divided into the union of countable and disjoint cubes satisfying

(1) there exists a sequence of parallel subcubes $\{K_l\}_{l=1}^\infty$ such that
\begin{equation}
t<\frac{1}{m(K_l)}\int_{K_l} |f|<2^n t;
\end{equation}

(2) \begin{equation}
|f|\le t, a.e. \ {\rm on}\ G=\R^n\setminus \cup K_l.
\end{equation}

Denote $F=\cup K_l$. Then it follows that that $m(F)\le \frac{\|f\|_1}{t}$. Let $f=g+b$, where $g$ is defined by
\begin{equation}\label{3.5}
g(x)=
\left\{
\begin{array}{ll}
&f(x), x\in G,\\
&\frac{1}{m(K_l)}\int_{K_l} f, x\in K_l, l=1,2,\cdots
\end{array}
\right.
\end{equation}
 and $b=f-g$ satisfies
\begin{equation}\label{3.6}
b(x)=0, x\in G, \quad \int_{K_l} b=0, l=1,2, \cdots.
\end{equation}

It is easy to get
\begin{equation}
\|g\|_q\le \|f\|_q, \quad 1\le q\le \infty.
\end{equation}

 Let  $\delta_l>0$ be the diameter of $K_l$ and $B_l\supset K_l$ be a ball with radius $\delta_l$. Denote $F^*=\cup B_l, G^*=\R^n\setminus F^*$. Then it yields
\begin{equation}\label{916}
m(F^*)\le n^{\frac n2}\omega_nm(F)\le \frac{C(n)\|f\|_1}{t}.
\end{equation}
The operator $T_1$ can be decomposed into
\begin{equation}\label{3.7}
\begin{split}
T_1f=&T_1g+T_1b\mathrm{I}_{G^*}+T_1b\mathrm{I}_{F^*}\\
=&(T_1g+T_1b\mathrm{I}_{F^*})+T_1b\mathrm{I}_{G^*}\\
\equiv&T_{11}f+T_{12}f,
\end{split}
\end{equation}
where $\mathrm{I}_A$ is the characteristic function on a set $A$, that is, $\mathrm{I}_A=1$ for $x\in A$ and $\mathrm{I}_A=0$ for $x\in \R^n\setminus A$.

Thanks to Lemma \ref{uni-est21}, we have
\begin{equation}
\begin{split}
\|T_{11}f\|_2\le & \|T_{11}g\|_2+\|T_{11}b\mathrm{I}_{F^*}\|_2\\
\le & C\|g\|_2+\|b\|_2.
\end{split}
\end{equation}
Direct estimates give
\begin{equation}
\begin{split}
\|g\|_2^2=&\int_G |f|^2 dx +\sum_l\int_{K_l} (\frac{1}{|K_l|}\int_{K_l} f)^2 dx\\
\le& \int_G |f|^2 dx +\sum_l\frac{1}{|K_l|}(\int_{K_l}f dx)^2\\
\le& \int_G |f|^2 dx +\sum_l\int_{K_l} |f|^2 dx=\|f\|_2^2.
\end{split}
\end{equation}
Hence $\|g\|_2\le \|f\|_2$. Combining the fact that $\|b\|_2\le \|f\|_2+\|g\|_2\le 2\|f\|_2$, we obtain
\begin{eqnarray}\label{970-}
\|T_{11}f\|_2\le C\|f\|_2,
\end{eqnarray}
which implies that $T_{11}$ is of type $(2,2)$.
Note that, for any $t>0$,
\begin{equation}\label{970}
\begin{split}
m\{x:|T_{11}f|>t\}&\le m\{x: |T_{11}g|>\frac t2\}+m\{x: |T_{11}b\mathrm{I}_{F^*}|>\frac t2\}.
\end{split}
\end{equation}
Since
\begin{equation*}
\begin{split}
\|g\|_2^2&=\int_{\R^n} |g(x)|^2 dx=\int_F |g(x)|^2 dx+\int_G |g(x)|^2 dx\\
&\le 2^{2n}t^2m(F)+t\int_G |f(x)| dx\\
&\le C(n)t\|f\|_1,
\end{split}
\end{equation*}
we obtain
\begin{equation}\label{971}
\begin{split}
m\{x: |T_{11}g|>\frac t2\}\le \frac{C(n)}{t^2}\|g\|_2^2\le \frac{C(n)}{t}\|f\|_1.
\end{split}
\end{equation}
It follows from \eqref{916} that
\begin{equation} \label{972}
m\{x: |T_{11}b\mathrm{I}_{F^*}|>\frac t2\}\le m(F^*)\le \frac{C(n)\|f\|_1}{t}.
\end{equation}
Substitute \eqref{971} and \eqref{972} into \eqref{970} to yield
\begin{equation}\label{973}
\begin{split}
m\{x:|T_{11}f|>t\}&\le \frac{C(n)\|f\|_1}{t},
\end{split}
\end{equation}
which implies that $T_{11}$ is of weak-type $(1,1)$.  Therefore, due to \eqref{970-} and \eqref{973}, by  Lemma \ref{912} and duality method, we obtain
\begin{eqnarray}\label{992}
\|T_{11}f\|_p\le C \|f\|_p, 1<p<\infty.
\end{eqnarray}

Now we estimate $T_{12}$.
Define
$$
b_l=
\left\{
\begin{array}{ll}
b, &x\in K_l,\\
0, & x\notin K_l.
\end{array}
\right.
$$
Then $b=\sum_{l=1}^\infty b_l$. Noticing that, for any $x\in \R^n\backslash K_l$,
\begin{equation}
\begin{split}
|Tb_l|=&|\int_{K_l} (\frac{x_j-y_j}{|x-y|^{n+1-\beta}}-\frac{x_j-\bar y_j}{|x-\bar y|^{n+1-\beta}}) b_l dy |\\
\le& C\delta_l\int_{K_l}\frac{1}{|x-y|^{n+1-\beta}}|b_l(y)| dy\\
=&C\delta_l\int_{R^n}\frac{|b_l(y)|\mathrm{I}_{K_l}(y)}{|x-y|^{n+1-\beta}} dy,
\end{split}
\end{equation}
where $\bar y$ is the center  and $\delta_l$ is the diameter of  the cube $K_l$ respectively,
we obtain
\begin{equation}\label{3.17}
\begin{split}
(\int_{\R^n\backslash B_l} |Tb_l|^q dx)^{\frac1q}\le &C\delta_l(\int_{\R^n\backslash B_l}|\int_{\R^n} \frac{|b_l(y)|\mathrm{I}_{K_l}(y)}{|x-y|^{n+1-\beta}} dy|^q dx)^{\frac1q}\\
\le &C\delta_l (\int_{\R^n}\mathrm{I}_{\R^n\backslash B_l}(x)|\int_{\R^n} \frac{|b_l(y)|\mathrm{I}_{K_l}(y)}{|x-y|^{n+1-\beta}} dy|^q dx)^{\frac1q}\\
\le& C\delta_l(\int_{\R^n}|\int_{\R^n}\frac{\mathrm{I}_{\{|x-y|\ge\frac{\delta_l}{2}\}}|b_l(y)|}{|x-y|^{n+1-\beta}} dy|^q dx)^{\frac1q}\\
\le& C\|b_l\|_{p;K_l},
\end{split}
\end{equation}
where $1\le p,q<\infty$ satisfying $\frac 1q=\frac 1p-\frac \beta n$ and the Young inequality has been used in the last inequality.

Moreover, it holds that
\begin{equation*}
\begin{split}
\|Tb\mathrm{I}_{G^*}\|^q_q\le &\sum_l\|Tb_l\|^q_{q;\R^n\backslash B_l}\le C\sum_l\|b_l\|^q_{p;K_l}\le C(\sum_l\|b_l\|^p_{p;K_l})^{\frac qp}\\
=& C\|b\|^q_{p;F}\le C\|f\|_p^q
\end{split}
\end{equation*}
for any $1<p\le q<\infty$ satisfying $\frac 1q=\frac 1p-\frac{\beta}{n}$. It follows that $\|Tb\mathrm{I}_{G^*}\|_q\le C\|f\|_p$ and
\begin{equation*}
\begin{split}
\|T_{12}f\|_q&=\|T_1b\mathrm{I}_{G^*}\|_q\\
&\le \|Tb\mathrm{I}_{G^*}\|_q+\|T_2b\mathrm{I}_{G^*}\|_q\\
&\le C(\|f\|_p+\frac{\beta^{\frac{n(q-1)}{q}}}{(n(q-1)-\beta q)^{\frac 1q}}\|f\|_1),
\end{split}
\end{equation*}
where $1<p\le q<\infty$ satisfying $\frac 1q=\frac 1p-\frac{\beta}{n}$. Lemma \ref{uni-est32+} is then proved and the proof of Theorem \ref{uni-est3} is finished.
\end{proof}

In the end, we prove Theorem \ref{uni-est30} as follows.

\begin{proof}[ Proof of Theorem \ref{uni-est30}]
According to \eqref{3.7}, the operator $T_1$ can be decomposed into
\begin{equation}\label{3.71}
\begin{split}
T_1f=&T_1g+T_1b\mathrm{I}_{G^*}+T_1b\mathrm{I}_{F^*}\\
=&(T_1g+T_1b\mathrm{I}_{F^*})+T_1b\mathrm{I}_{G^*}\\
\equiv&T_{11}f+T_{12}f,
\end{split}
\end{equation}
Thanks to \eqref{973}, one has
\begin{equation}\label{9731}
\begin{split}
m\{x:|T_{11}f|>t\}&\le \frac{C(n)\|f\|_1}{t},
\end{split}
\end{equation}
Concerning $T_{12}$, we take $p=1$ in \eqref{3.17} to obtain
\begin{equation}
\begin{split}
\|Tb\mathrm{I}_{G^*}\|_q\le&\sum_l\|Tb_l\|_{q;\R^n\backslash B_l}\\
\le & C\sum_l\|b_l\|_{1;K_l}\\
=& C\|b\|_{1;F}\le C\|f\|_1
\end{split}
\end{equation}
for $q=\frac{n}{n-\beta}$. Then, in view of Lemma \ref{uni-est1},
$$
\|T_1b\mathrm{I}_{G^*}\|_q\le \|Tb\mathrm{I}_{G^*}\|_q+\|T_2b\mathrm{I}_{G^*}\|_q\le C\|f\|_1,
$$
 which implies that $T_{12}$ is of type  $(1,q)$, where $\frac 1q=1-\frac \beta n=\frac{n-\beta}{n}$.
 It concludes that, for any $t>0$,
 \begin{eqnarray*}
 m\{x:|T_{1}f|>t\}&\le m\{x: |T_{11}f|>\frac t2\}+m\{x: |T_{12}f|>\frac t2\}\\
 &\le C(\frac{\|f\|_1}{t}+\frac{\|f\|^q_1}{t^q}).
 \end{eqnarray*}
The proof of Theorem \ref{uni-est30} is proved.
\end{proof}

{\bf Acknowledgements.}
 Jiu is partially supported  by the National Natural Science Foundation of China
(No.11671273).

\end{document}